\pdfoutput=1
\RequirePackage{ifpdf}
\ifpdf % We~are running pdfTeX in pdf mode
\documentclass[pdftex]{sigma}
\else
\documentclass{sigma}
\fi

\numberwithin{equation}{section}

\newtheorem{thm}{Theorem}[section]
\newtheorem{prop}[thm]{Proposition}
\newtheorem{lem}[thm]{Lemma}

 { \theoremstyle{definition}
 \newtheorem{dfn}[thm]{Definition}
\newtheorem{rmk}[thm]{Remark}
\newtheorem{ex}[thm]{Example}
}

\newcommand{\reals}{{\mathbb R}}
\newcommand{\integers}{{\mathbb Z}}
\newcommand{\Hom}{\operatorname{Hom}}
\newcommand{\qbar}{{\overline{Q}}}
\newcommand{\xbar}{{\overline{X}}}
\newcommand{\ybar}{{\overline{Y}}}
\newcommand{\zbar}{{\overline{Z}}}
\newcommand{\calc}{{\cal C}}
\newcommand{\cale}{{\cal E}}
\newcommand{\calk}{{\cal K}}
\newcommand{\calt}{{\cal T}}
\newcommand{\boldc}{{\mathbf C}}
\newcommand{\from}{\leftarrow}
\newcommand{\rel}{\mathbf{REL}}
\newcommand{\slrel}{\mathbf{SLREL}}
\newcommand{\gslrel}{\it{G}\mathbf{SLREL}}
\newcommand{\unitob}{\mathbf{1}}
\newcommand{\xfy}{X \stackrel{f}{\longleftarrow}Y}
\newcommand{\wwxfy}{X \stackrel{[f]}{\longleftarrow}Y}
\newcommand{\ygz}{Y \stackrel{g}{\longleftarrow}Z}

\begin{document}

\allowdisplaybreaks

\newcommand{\arXivNumber}{2408.06363}

\renewcommand{\PaperNumber}{101}

\FirstPageHeading

\ShortArticleName{The Wehrheim--Woodward Category of Linear Canonical Relations between $G$-Spaces}

\ArticleName{The Wehrheim--Woodward Category \\ of Linear Canonical Relations between $\boldsymbol{G}$-Spaces}

\Author{Alan WEINSTEIN~$^{\rm ab}$}

\AuthorNameForHeading{A.~Weinstein}

\Address{$^{\rm a)}$~Department of Mathematics, University of California, Berkeley, CA 94720, USA}
\EmailD{\href{mailto:alanw@math.berkeley.edu}{alanw@math.berkeley.edu}}
\Address{$^{\rm b)}$~Department of Mathematics, Stanford University, Stanford, CA 94305, USA}

\ArticleDates{Received August 18, 2024, in final form November 15, 2024; Published online November 18, 2024}

\Abstract{We extend the work in a previous paper with David Li-Bland to construct the Wehrheim--Woodward category WW($G\mathbf{SLREL}$) of equivariant linear canonical relations between linear symplectic $G$-spaces for a compact group $G$. When $G$ is the trivial group, this reduces to the previous result that the morphisms in WW($\mathbf{SLREL}$) may be identified with pairs $(L,k)$ consisting of a linear canonical relation and a nonnegative integer.}

\Keywords{symplectic vector space; canonical relation; rigid monoidal category; highly selective category}

\Classification{53D05; 18F99}

\section{Introduction}%\label{sec-intro}
The problem of making canonical relations between symplectic manifolds into the morphisms of a~category was solved in a seminal paper by Wehrheim and Woodward \cite{we-wo:functoriality}.
In a paper \cite{li-we:selective} by David Li-Bland and the author, we gave an abstract version of the Wehrheim--Woodward construction and showed that the morphisms in the WW-category of linear canonical relations between finite-dimensional symplectic vector spaces could be identified with the pairs $(L,k)$ consisting of a~linear canonical relation and a nonnegative integer, called {\em indexed canonical relations}.\looseness=-1

The main result of the present paper is an extension of this result to
the situation where a~compact group $G$ is acting on the linear symplectic spaces, and the linear canonical relations are equivariant. In this case, the nonnegative integer $k$ is replaced by an arbitrary isomorphism class $\cale$ of finite-dimensional linear $G$-spaces (without symplectic structure involved).

We recall the general setting of \cite{li-we:selective}, to which we will refer frequently.
The basic abstract idea was to select, within an underlying category $\boldc$, a collection of
nice morphisms, called {\it suave}, and a collection of composable pairs of suave morphisms, called {\it congenial}, such that every pair including an identity morphism is congenial, and such that the composition of a congenial pair is always suave. From such a {\it selective category} $\boldc$, we formed the category WW$(\boldc)$ generated by the suave morphisms, with relations given by the congenial compositions. For certain purposes, we also distinguished subcategories of the suave morphisms whose members are called {\it reductions} and {\it coreductions}. After the imposition of some further axioms, notably the requirement that each suave morphism be the composition of a reduction with a coreduction,
we obtained the definition of a {\it highly selective category}.

 We extended to all highly selective categories $\boldc$ the main result of \cite{we:note} to the effect that, in the special case where $\boldc$ is a certain highly selective category of relations between symplectic manifolds, with suave morphisms the smooth canonical relations, every morphism\footnote{We will usually denote morphisms in categories by arrows $X\from Y$ from right to left; the notation $\Hom(X,Y)$ will therefore denote morphisms {\em to} $X$ {\em from} $Y$.} $X \from Y$ in~WW$(\boldc)$ may be represented by a composition (not necessarily congenial)
$X \twoheadleftarrow Q \leftarrowtail Y$ of just two suave morphisms in $\calc$, where the decorations on the arrows mean that $X \twoheadleftarrow Q$ is a~reduction and $Q \leftarrowtail Y$ is a coreduction.

We paid special attention to rigid monoidal
structures, with monoidal products denoted $X\otimes Y$, dual objects $\xbar$, and
unit object $\unitob$. When a rigid monoidal structure is compatible with a selective structure on $\boldc$, it extends to WW$(\boldc)$.
In any rigid monoidal category, the morphisms from $\unitob$ play a special role. In fact, $\Hom(X,Y)$ can be identified with $\Hom\bigl(X \otimes \ybar, \unitob\bigr)$, with each morphism $X\from Y$ corresponding to its ``graph" $X\otimes \ybar\from\unitob$. In WW$(\boldc)$, therefore, each morphism may be represented as a composition $X\otimes \ybar \twoheadleftarrow Q \leftarrowtail \unitob$.

The morphisms in WW-categories built from categories of relations will be called {\it hyperrelations}. We will sometimes refer to a diagram $X\otimes \ybar \twoheadleftarrow Q \leftarrowtail \unitob$ in any category as a~{\it hypermorphism} to $X$ from $Y$. The composition
of hypermorphisms is essentially a monoidal product, which is
always defined without any extra assumptions, though one must remember the equivalence relation, for instance to see why an identity morphism really is an identity.

\section[Linear hypercanonical relations between symplectic G-spaces]{Linear hypercanonical relations between symplectic $\boldsymbol{G}$-spaces}
In this section, we analyze in detail the WW category built from the highly selective rigid monoidal category $\gslrel$
 of finite-dimensional symplectic $G$-vector spaces and equivariant linear canonical relations, where $G$ is a group. Beginning with Lemma \ref{lemma-complements}, we usually assume that $G$ is compact. This generalizes \cite[Section 7]{li-we:selective}, which deals with the special case $\slrel$, where $G$ is the trivial group.

The monoidal product in $\gslrel$ is the Cartesian product, so we will denote it by
$\times$ rather than $\otimes$. The unit object $\unitob$ is the zero-dimensional vector space whose only element is the empty~set.

 The dual $\xbar$ of $X$ is the same vector space, but with its symplectic structure multiplied by $-1$. The morphisms $X\from Y$ are the $G$-invariant Lagrangian subspaces of $X\times \ybar$.\footnote{Here, the identification of morphisms with their graphs is essentially tautological.} In particular, the morphisms $X\from \unitob$ are just the invariant Lagrangian subspaces of $X$; similarly, the morphisms $\unitob \from X$ are the invariant Lagrangian subspaces of $\xbar$, but these are the same as the invariant Lagrangian subspaces of $X$. Thus, the unit $\delta_X$ and counit $\epsilon_X$ are both given by the diagonal $\{(x,x) \mid x \in X\}$, an invariant Lagrangian subspace of $X \times \xbar$.

We will use the selective structure in which all morphisms are suave, but only the monic pairs are congenial.
We recall from \cite{li-we:selective} that monicity for relations \smash{$\xfy$} and \smash{$\ygz$} is defined as injectivity of the projection from $(f\times g )\cap \bigl( X \times \Delta_Y \times \zbar\bigr)$ to $X\times \zbar$. For linear relations, this is equivalent to injectivity over $0$, i.e., the condition
 ${(f \times g )\cap \bigl(\{0_X\} \times \Delta_Y \times \bigl\{0_\zbar\bigr\}\bigr) = \{0\}}$. By~elementary symplectic linear algebra, this is equivalent to $f \times g$ being transversal to ${X \times \Delta_Y \times \zbar}$, i.e., transversality of the composition. For the highly selective structure, we define the reductions to be those morphisms which are surjective and single valued and the coreductions those which are injective and everywhere defined, just as in $\slrel$. (In $\gslrel$, the domain of a~reduction and the codomain of reduction are invariant in addition to being coisotropic.)

\begin{prop}
%\label{slrelselectiverigidity}
The category $\gslrel$ with the structures described above is a highly selective rigid monoidal category.
\end{prop}

\begin{proof}
Most of the required properties follow from those of $\rel$, as demonstrated in \cite[Exam\-ples~2.4,~3.2,~4.1 and~4.10]{li-we:selective}. Morphisms to and from the zero vector space are clearly reductions and coreductions respectively, and application of
\cite[Proposition 4.12]{li-we:selective} gives the factorization of suave morphisms.
\end{proof}

Since the congenial pairs in $\gslrel$ are monic, there is a good notion of trajectories for the WW morphisms. Given $[f]$, the equivalence class in ${\rm WW}(\gslrel)$ of a sequence $f_i$ of morphisms in $\gslrel$, if $(x,y)$ is in the shadow of $[f]$, the set theoretic composition of the $f_i$, the set of trajectories to $x$ from $y$ is an affine $G$-space modeled on the vector space of trajectories to $0$ from $0$, which is the kernel of a projection from a fibre product.
That the structure of this kernel as a linear $G$-space is an invariant is a consequence of the invariance of spaces of trajectories, with the linear $G$-structure taken into account. This justifies the following definition.

\begin{dfn}
%\label{dfn-excess}
Let \smash{$\wwxfy$} be a morphism in ${\rm WW}(\gslrel)$. The {\it excess} $\cale([f])$ of $[f]$ is the isomorphism class of the affine $G$-space of trajectories between any two points in~$X$ and~$Y$.
\end{dfn}

We omit the easy proof of the following.
\begin{prop}
A pair $(f,g)$ is congenial if and only if the excess of $[fg]$ is a point. More generally, a~WW morphism is represented by a single linear canonical relation if and only if its excess is a~point.
\end{prop}

To classify all of the WW morphisms, it is useful to begin with the case $Y = \unitob$.
A {\it hyper-Lagrangian subspace} of the symplectic $G$-vector space $X$ is defined as a WW morphism $X\from \unitob$. It~is~represented by (equivariant) diagrams of the form
\[
X \stackrel {C}{\twoheadleftarrow}Q\stackrel{L}{\leftarrowtail}\unitob ,
\]
 where $C$ is a reduction whose domain, an invariant coisotropic subspace, will also be denoted by $C$, and $L$ is invariant and Lagrangian in $Q$. We may therefore denote the hyper-Lagrangian subspace by $[C,L]$. The set-theoretic composition $CL$, an invariant Lagrangian subspace of $X$, is~the~shadow of $[C,L]$ and is therefore well defined. The excess of $[C,L]$ is the isomorphism class of $C\cap L$ as a linear $G$-space or, equivalently, that of $Q/(C+L)$.

The following lemma may be well known, but we give a proof here for completeness. (Beginning here, we will usually assume that the group $G$ is compact.)

\begin{lem}
\label{lemma-complements}
Suppose that a compact group $G$ acts linearly on a vector space $V$. Let $L$ and $J$ be invariant subspaces such that $L\cap J = \{0\}$. Then $J$ is contained in an invariant complement~$M$ to~$L$. $($In particular, if $J=\{0\}$, this shows that every invariant subspace has an invariant complement$.)$ Furthermore, if $V$ is a symplectic $G$-space, $L$ is Lagrangian, and $J$ is isotropic, then $M$ can be chosen to be Lagrangian.
\end{lem}

\begin{proof}
By elementary linear algebra, $J$ is contained in a complement $M'$ to $L$, which may or not be invariant. To make it invariant, we may average over $G$. In fact, each complement~$N$ to~$L$ in $V\cong L\oplus M'$ is the graph of a linear map $\lambda(N)$ from $M'$ to $L$, and the map $\lambda$ is a bijection between the linear maps $M'\to L$ and the complements to $L$. Now we may define $M$ to be
\begin{equation}
\label{eq-average}
\lambda^{-1}\int_G \lambda\bigl(g\cdot M'\bigr) {\rm d}g,
\end{equation}
where ${\rm d}g$ is the left-invariant measure on $G$ with total measure 1.
The fact that $J$ is invariant and contained in $M'$ implies that $J$ is also contained in each $g\cdot M'$, so that each $\lambda(g\cdot M')$
in equation~\eqref{eq-average} vanishes on $J$. It follows that the integral also vanishes, so that $J$ is contained in~$M$.

For the symplectic case, we can assume that the complement $M'$ to $L$ is also Lagrangian. To see this, start with a Lagrangian complement $M''$ to $L$ which does not necessarily contain~$J$. Like any Lagrangian complement, $M''$ is naturally isomorphic to $L^*$. Via this isomorphism, $\lambda$~(defined as was done above for $M'$) maps the \textit{Lagrangian} complements of $L$ onto the \textit{symmetric} maps $L\to L^*$. What about \textit{isotropic} subspaces $J$ for which $L\cap J = \{0\}$? Each such subspace corresponds to the image of a map $r$ to
$L^*$ from a subspace $K$ of $L$ which is symmetric in the sense that $(r(x))(y)=(r(y))(x)$ for all $x$* and $y$ in $K$. It is easy to see that such a map $r$ can be extended to a symmetric map from $L$ to $L^*$ which corresponds to the desired Lagrangian extension $M'$ to $J$.

Since the action of $G$ on $V$ is symplectic, each $g\cdot M'$ is also Lagrangian, so the integral in~\eqref{eq-average} is symmetric, and $M$ is Lagrangian as well.
\end{proof}

\begin{lem}
\label{lemma-isotropiclagrangian}
When $G$ is compact, a symplectic $G$-vector space $V$,
an invariant Lagrangian subspace $L$, and an invariant isotropic subspace $I$ are characterized up to equivariant symplectomorphism by the isomorphism classes of the linear $G$-spaces $L$, $I$ and $L\cap I$.
\end{lem}

\begin{proof}
 Choose invariant complements to write $I = (I\cap L) \oplus J$ and $L = (I\cap L)\oplus K$. Then~${J\cap L = \{0\}}$, and $J$ is isotropic, so, by Lemma~\ref{lemma-complements}, $J$ can be extended to an invariant Lagrangian complement $M$ of $L$ in $V$. We may identify $M$ with $L^*$ and hence identify $V$ with the direct sum~${L\oplus L^*}$ with the usual symplectic structure. The decomposition of $L$ gives, via the orthogonality relation $^o$ between subspaces of $L$ and those of $L^*$, a dual decomposition
$L^* = K^o \oplus (I\cap L)^o$, which is naturally isomorphic to $(I\cap L)^* \oplus K^*$. Since $J\subset L^*$ is symplectically orthogonal to~${I\cap L}$, it must be contained in the summand $K^*$.

Now choose an invariant complement $R$ to $J$ in $K^*$. Since $K^*=J\oplus R$, we have $K=J^*\oplus R^*$. This gives $L=(I\cap L)\oplus J^* \oplus R^*$, and hence $L^*=(I\cap L)^* \oplus J \oplus R$. Thus, $V=L\oplus L^*$ can be written as $(I\cap L)\oplus(I\cap L)^*\oplus J^* \oplus J \oplus R^* \oplus R$, in which $I=(I\cap L)\oplus J$. Since $J$ is isomorphic to $I/(I\cap L)$ and $R$ to $L/(I\cap L)$, the triple $(V,L,I)$ is determined up to symplectic $G$-isomorphism by the $G$-isomorphism types of $L$, $I$ and $L\cap I$.
\end{proof}

\begin{prop}
When the group $G$ is compact, two hyper-Lagrangian subspaces are equal if they have the same shadow and the same
excess.
\end{prop}

\begin{proof}
Let $\Lambda$ be an invariant Lagrangian subspace of $X$, $\calk$ an isomorphism class of $G$-vector spaces, and $E$ a representative of $\calk$. We will construct a normal form which is equivalent to any representative of a hyper-Lagrangian subspace with
shadow $\Lambda$ and excess $\calk$. In $X \times (E \oplus E^*)$, let $C_{\calk,0}= X \times E$ and $\Lambda_{\calk,0} = \Lambda \times E$.
 The hyper-Lagrangian subspace $\langle C_{\calk,0}, \Lambda_{\calk,0}\rangle$ has shadow $\Lambda$ and excess $\calk$. This is a minimal representative in its equivalence class; we get larger representatives by forming its monoidal product with the trivial hyper-Lagrangian subspace of the point $\reals^0$, as represented as the transversal pair $(\reals^r,{\reals^r}^*)$ with intermediate space $\reals^r \times{\reals^r}^*$. Denote this product by $\langle C_{\calk,r}, \Lambda_{\calk,r}\rangle$.

Now let $\langle C,L \rangle$ be any hyper-Lagrangian subspace of $X$ with shadow $\Lambda$, excess $\calk$, and intermediate space $Q$ of dimension $2N$. If $X$ has dimension $2n$, then the dimension of $C$ must be~${N+n}$. We observe first that the diagram
\[
X \stackrel {C}{\twoheadleftarrow}Q\stackrel{L}{\leftarrowtail} \mathbf 1
\]
 is symplectically isomorphic to
\[
X \stackrel {C_{\calk,r}}{\twoheadleftarrow}X\times E \times E^*\times\reals^r \times{\reals^r}^*\stackrel{\Lambda_{\calk,r}}{\leftarrowtail} \mathbf 1 ,
\]
 i.e., there is a symplectomorphism of $Q$ with $X \times E \times E^* \times \reals^r \times{\reals^r}^*$ with $r=N-n-k$, taking $L$ to~$\Lambda \times \reals^k \times{ \reals^r}^*$ and $C$ to $X\times \reals^k\times {\reals^r}$. The isomorphism now follows from Lemma~\ref{lemma-isotropiclagrangian} above.
That $\langle C,L\rangle$ and $\langle C_{\calk,r}, \Lambda_{\calk ,r}\rangle$ are equal as WW-morphisms now follows from \cite[Proposition~5.1]{li-we:selective}.
\end{proof}

Thus, there is a bijective correspondence between hyper-Lagrangian subspaces of $X$ and pairs $(L,\calk)$, where $L$ is an ordinary Lagrangian subspace and $\calk$ is an isomorphism class of linear $G$-spaces.

We can now understand general WW morphisms via their graphs. We will call any morphism
\smash{$\wwxfy$} in ${\rm WW}(\gslrel)$ a ({\it linear}) {\it hypercanonical relation} to $X$ from $Y$. Its graph is a~hyper-Lagrangian subspace of $X\times \ybar$. These have the following two useful properties, which hold even if $G$ is not compact.

\begin{prop}\quad
%\label{prop-hypercanonical}
\begin{enumerate}\itemsep=0pt
\item[$(1)$] The excess of any hypercanonical relation is equal to that of its graph.

\item[$(2)$] If $[f]$ and $[g]$ are composable hypercanonical relations,
\[
\cale([f][g]) = \cale([f]) \oplus \cale([g]) \oplus \cale([c([f]),(c[g])]).
\]
\end{enumerate}
\end{prop}

\begin{proof}
Given a morphism $[f]$ represented by \smash{$X \stackrel{a}{\twoheadleftarrow} Q \stackrel{b}{\leftarrowtail} Y$}, its excess is the isomorphism class of the trajectory $G$-space
\[
\smash{(a \times b)\cap (\{0_X\}\times \Delta_Q\times
\{0_Y\}) \subseteq X\times Q\times \qbar\times \ybar}.
\]

On the other hand, the graph of $[f]$ is represented by
\[
X\times \ybar \stackrel{a \times 1_\ybar}{\twoheadleftarrow} Q \times \ybar
 \stackrel{\gamma_b}{\leftarrowtail} \unitob.
\]
The excess of the latter is the isomorphism class of
\[
\bigl(a\times 1_\ybar \times \gamma_b\bigr) \cap\bigl({0_{X\times \ybar}} \times \Delta_{\qbar \times Y}
 \times\{0_\unitob\}\bigr)\subseteq X \times \ybar\times \qbar\times Y \times Q\times\ybar \times \unitob.
\]
This intersection consists of the sextuples $(x,y,q',y',q'',y'')$ such that
$x=0_X$, $y=0_Y$, $y=y'=y''$, $q'=q''$, $(x,q')\in a$, and $(q'',y'')\in b$. These may be
identified with the trajectories to $0_X$ from $0_Y$ in $[a,b]=[f]$.

Note that the graph of $[f]$ is also represented by the ``graph product"
\[
X\times \ybar \twoheadleftarrow X \times \qbar \times Q \times \ybar
\stackrel{\gamma_a \times \gamma_b}{\leftarrowtail} \unitob\times\unitob=\unitob,
\]
whose trajectory space is essentially the same as that of $[ab]$.

For (2), we begin with the fact that $\calt([f,g])$ is the fibre product over $Y$ of $\calt([f])$ and $\calt([g])$. The projection $\tau([f,g])$ may therefore be factored as
\[
c([fg])\stackrel{}{\twoheadleftarrow} c([f])\times_Y c([g]) \stackrel{}{\leftarrowtail} \calt([f,g]).
\]
 The kernel of the map to $c([fg])$ has isomorphism class
$\cale([c([f]),(c[g])])$, while the kernel of the map from $\calt([f,g])$ has isomorphism class
$\cale([f]) + \cale([g])$.
\end{proof}

We may therefore identify the WW morphisms (i.e., hypercanonical relations) $X\from Y$ with the pairs $(f,\calk)$, where $f$ is a Lagrangian subspace of $X \times \ybar$, and $\calk$ is an isomorphism class of linear $G$-spaces. To be consistent with the case of the trivial group $G$ studied in \cite{li-we:selective}, we shall continue to call these pairs {\it indexed canonical relations}.

The following theorem expresses the structure of ${\rm WW}(\gslrel)$ for compact $G$ via the identification
with indexed canonical relations. The proof consists of elementary calculations.

\begin{thm}
%\label{thm-indexed}
The indexed $($linear$)$ canonical relations form a category by identification with the linear hyperrelations. The composition law is
\[
\bigl(f',\calk'\bigr)\bigl(f'',\calk''\bigr) = \bigl(f'f'',\calk'+\calk''+\cale\bigl(\big[f',f''\big]\bigr)\bigr),
\]
 and the monoidal product is
\[
\bigl(f',\calk'\bigr)\times\bigl(f'',\calk''\bigr) = \bigl(f'\times f'',\calk'\oplus \calk''\bigr).
\]
The monoid of endomorphisms of the unit object is naturally identified with the set of isomorphism classes of finite-dimensional $G$-spaces with the operation of direct sum, and its action on the category is the shifting operation $\calk'\cdot (f,\calk)= (f,\calk'\oplus\calk)$.
 The trace of an endomorphism~$(f,\calk)$ of $X$ is
the direct sum of $\calk$ with the isomorphism class of the fixed point space of~$f$, whose diagonal is $ (\gamma_f \cap \Delta_X)$.
\end{thm}

\begin{rmk}
The trace of an endomorphism $f$ of $X$
is the endomorphism of the unit object~$\unitob$ defined as the composition
\smash{$\epsilon_{\overline X}\gamma_f$}, where $\epsilon_{\overline X}$ is the ``counit'' morphism to $\unitob$ from $X \times \overline X$ given by~the diagonal, and $\gamma_f$ to $X\times \overline X$ from $\unitob$ is the graph of $f$.\footnote{In \cite{li-we:selective}, $\epsilon_X$ is mistakenly substituted for $\epsilon_{\overline X}$ in the definition of the graph.}
\end{rmk}

\begin{ex}
Suppose that $G=S^1=\reals/2\pi\integers$. The isomorphism classes $\calk$ of linear $G$-spaces may be identified with finitely supported sequences of nonnegative integers $n_0,n_1,n_2,\dots$, where~$n_k$ is the multiplicity in any representative of $\calk$ of the irreducible representation $\theta\mapsto {\rm e}^{{\rm i}k\theta}$.
\end{ex}

\begin{ex}
 The fixed point set of a projection onto an invariant coiso\-tropic subspace $C$ is~$C$ itself.
\end{ex}

The referee has suggested two possible continuations of this work.
\begin{enumerate}\itemsep=0pt
\item The characterization of {$\gslrel$} might be related to an
equivariant version of normal forms for symplectic matrices \cite{gu:normal}. In particular,
it would be interesting to connect the different cases of the conjugacy class of
a symplectic matrix and hyper-Lagrangian subspaces.

\item In recent work \cite{co-me-st:frobenius}, it is suggested that the bicategory
SPAN is a higher-categorical version of the WW construction for the category REL of set-theoretic relations.
It would be interesting to see the 2-categorical analogue of $\gslrel$ and the
role of the isomorphism classes of linear $G$-spaces.
\end{enumerate}

\subsection*{Acknowledgements}
The author was partially supported by a grant from UC Berkeley.

\pdfbookmark[1]{References}{ref}
\LastPageEnding

\end{document}